\documentclass[]{article}
\usepackage{amssymb,amsmath,amsthm}
\usepackage{texdraw}
\usepackage{epic,eepic}

\newtheorem{thm}{Theorem}[section]
\newtheorem{prop}[thm]{Proposition}

\newtheorem{corollary}[thm]{Corollary}

\theoremstyle{remark}
\newtheorem{rem}[thm]{Remark}
\newtheorem{expl}[thm]{Example}
\date{}
\begin{document}
\title{Yang--Baxter maps with first--degree--polynomial $2 \times 2$ Lax matrices}
\author{Theodoros E. Kouloukas and Vassilios G. Papageorgiou
\footnote{E-Mail: tkoulou@master.math.upatras.gr, vassilis@math.upatras.gr}\\
Department of Mathematics, University of Patras, \\ GR-265 00 Patras, Greece}
\maketitle

\begin{abstract}
A family of non--parametric Yang--Baxter (YB) maps is constructed by
re--factorization of the product of two  $2 \times 2$ matrix polynomials of first degree. These maps are Poisson
with respect to the Sklyanin bracket. For each Casimir function a
parametric Poisson YB map is generated by reduction on the
corresponding level set. By considering a complete set of Casimir functions
symplectic multi--parametric YB maps are derived.
These maps are quadrirational with explicit formulae  in
terms of matrix operations. Their Lax matrices are, by construction, $2 \times 2$ first--degree--polynomial 
in the spectral parameter and are classified  by Jordan normal form of the leading term. Non--quadrirational 
parametric YB maps constructed as limits of the
quadrirational ones are connected to known integrable systems on quad--graphs.
\end{abstract}

\section{Introduction} \label{intro}

The question of finding set--theoretical solutions of the quantum
Yang-Baxter equation was first suggested by Drinfeld in \cite{DRIN}. 
Certain examples of such solutions had already appeared in the 
relevant literature by Sklyanin \cite{SKL88}. The
dynamical aspects of these solutions were studied by Veselov in
\cite{VES03} where the short term ``Yang-Baxter maps'' was proposed for them.
Recent results \cite{PTV, PT07} connect these solutions with integrable equations on
quad-graphs through symmetry--reduction. Actually the connection 
between the YB relation for maps and the  {\em multidimensional consistency} property 
 \cite{N,BS} for discrete equation on quad--graphs was already noticed by Adler, Bobenko and Suris
(see concluding remarks of \cite{ABS1}). They also gave a classification of YB maps on $\mathbb{CP}^{1} \times
\mathbb{CP}^{1}$ in \cite{ABS2}.
Weinstein and Xu \cite{WX} found a way of constructing YB maps (classical solutions of the quantum YB equation, in their terminology)
using the theory of Poisson--Lie groups. Their method was generalised in \cite{LYZ}. 
The algebraic theory of YB maps was developed by Etingof, Schedler and 
Soloviev \cite{ESS}. It seems though that dressing transformations connected to soliton equations
and associated constructions involving
loop groups are giving easily many low dimensional examples of YB maps as well as the most simple and fundamental  
parametric one i.e.  Adler's map. The constrtuction of the latter in \cite{RV} 
was given  by Hamiltonian reduction of the loop group $LGL(2,\mathbb{R})$
equiped with the Sklyanin bracket \cite{SKL83}. For a review  on YB maps one can look at \cite{VES07}. 

Based on these ideas we present in this work a construction of
symplectic, parametric YB maps on $\mathbb{C}^{2} \times
\mathbb{C}^{2}$ with  $2 \times 2$ first--degree--polynomial Lax  matrices. 
In section 2 we give the necessary definitions and notation. Section 3 contains the
construction of a non--parametric quadrirational YB map from a
re--factorization procedure. The proof of its YB property is put in an appendix.
This map is Poisson with respect to the Sklyanin bracket presented in section 4.
A reduction procedure to symplectic leaves
is also applied in order to obtain symplectic parametric YB maps and
their corresponding Lax matrices. A classification is provided by
Jordan normal forms. In section 5 non--quadrirational YB maps are derived as
limits of the quadrirational ones of the previous section. 
We finally conclude in section 6 giving some perspectives for future work.
\section{Yang-Baxter Maps and Lax Matrices} \label{sec1}

A {\em set--theoretic solution of the Quantum Yang--Baxter
equation} \cite{DRIN}, or just a {\em Yang-Baxter (YB) Map} \cite{VES03} is a map $R:
\mathcal{X} \times \mathcal{X} \rightarrow \mathcal{X} \times
\mathcal{X}$, where $\mathcal{X}$ is any set, which satisfies the
equation:
\begin{equation}
R_{23}R_{13}R_{12}=R_{12}R_{13}R_{23} \label{YBprop}
\end{equation}
where $R_{ij}$ for $i,j=1,...,3$ is the map that acts as $R$ on the
$i$ and $j$ factors of $\mathcal{X} \times \mathcal{X} \times
\mathcal{X}$ and identically on the rest. 
In various  examples of YB maps, e.g. maps arising from geometric 
crystals \cite{ETING}, the set $X$ has the structure of an algebraic variety 
and $R$ is a birational isomorphism. We are also concerned with birational YB maps here as well.
 A Yang-Baxter map
$R:(x,y)\mapsto (u,v)=(u(x,y),v(x,y))$ is called {\em
quadrirational} \cite{ABS2} if the maps
$u(\cdot,y):\mathcal{X}\rightarrow\mathcal{X}$ and
$v(x,\cdot):\mathcal{X}\rightarrow\mathcal{X}$ are bijective
rational maps.

A {\em parametric YB map} is a YB map
$$R:((x,\alpha),(y,\beta))\mapsto((u,\alpha),(v,\beta))=
(u(x,\alpha, y,\beta),v(x,\alpha, y,\beta))$$  where $x, \ y \in
\mathcal{X}$ and the parameters $\alpha, \beta \in \mathbb{C}^n$. It
is useful to keep the parameters separately and denote
$R(x,\alpha,y,\beta)$ by $R_{\alpha,\beta}(x,y)$. A {\em Lax Matrix}
for this map is a matrix $L(x,\alpha, \zeta )$ that depends on the
point $x$, the parameter $\alpha$ and a spectral parameter $\zeta$
(we usually denote it just by $L(x;\alpha)$), such that
\begin{equation} \label{laxmat}
L(u;\alpha)L(v;\beta)=L(y;\beta)L(x;\alpha),
\end{equation}
for any $\zeta\in \mathbb{C}$. Here we have adopted the definition of a
Lax matrix from  \cite{SV} but we have to notice
that this definition does not imply necessary that equation
(\ref{laxmat}) is equivalent to $(u,\ v)=R_{\alpha,\beta}(x,y)$.

We can represent any parametric YB map with an elementary quadrilateral like in Fig.1.

\vspace{1.2cm}

\centertexdraw{ \setunitscale 0.15 \linewd 0.01
\move (7 1) \linewd 0.05 \arrowheadtype t:F \avec(1 7) \lpatt( )
\move (0 0) \linewd 0.1 \lvec(8 0)  \lvec (8 8)  \lvec (0 8)  \lvec (0 0) \lpatt( )
\htext (2.8 -1.5){$(x;\alpha)$} \htext (8.5 3.5){$(y;\beta)$}
\htext (2.8 8.5){$(u;\alpha)$} \htext (-3.3 3.5){$(v;\beta)$}
\htext (4.3 3.8){$R_{\alpha,\beta}$}
}
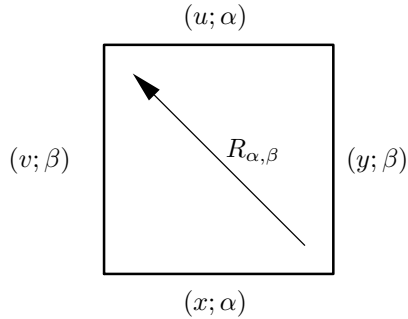
\begin{center} {\em Fig.1 A map assigned to the edges of a quadrilateral}
\end{center}

Let
$R_{23}R_{13}R_{12}(x,y,z)=(x^{\prime \prime },y^{\prime \prime },z^{\prime
\prime })$ and
$R_{12}R_{13}R_{23}(x,y,z)=(\tilde{\tilde{x}}, \tilde{\tilde{y}},\tilde{\tilde{z}})$.
We can represent these maps as chains of maps at the faces of a cube like in Fig.2.

\vspace{1cm}
\centertexdraw{ \setunitscale 0.15 \linewd 0.01
\move (9.8 2.4) \linewd 0.05 \avec(8.5 7.5) \lpatt( )
\move (7 0.5) \linewd 0.03 \arrowheadtype t:F \avec(1.9 1.6) \lpatt( )
\move (0 0) \linewd 0.01 \lvec(8 0)  \lvec (8 8)  \lvec (0 8)  \lvec (0 0) \lpatt( )
\move (8 0) \linewd 0.01 \lvec(8 8)   \lvec (10 10)  \lvec (10 2) \lvec (8 0)
\move (0 8) \linewd 0.01 \lvec (2 10)   \lvec (10 10)  \lvec (8 8)  \lvec (0 8)\lpatt( )
\move (0 0) \linewd 0.01       \lpatt (0.2 0.2) \lvec(2 2) \lvec(10 2)
\move (0 0) \linewd 0.01       \lpatt (0.2 0.2) \lvec(2 2) \lvec(2 10)

\htext (4 -1.1){$x$} \htext (9.2 0.2){$y$}  \htext (10.3 5.5){$z$}
\htext (6 2.2){$x'$} \htext (0.5 1.3){$y'$}  \htext (1.2 5.3){$z'$}
\htext (4.5 10.5){$x''$} \htext (-0.5 8.5){$y''$}  \htext (-1.5 5.1){$z''$}
\htext (4 8.2){$\tilde{x}$} \htext (8.2 8.8){$\tilde{y}$}  \htext (7.1 4){$\tilde{z}$}
\htext (6.9 10.5){$\tilde{\tilde{x}}$} \htext (0.7 9.7){$\tilde{\tilde{y}}$}  \htext (-1.5 3){$\tilde{\tilde{z}}$}
}
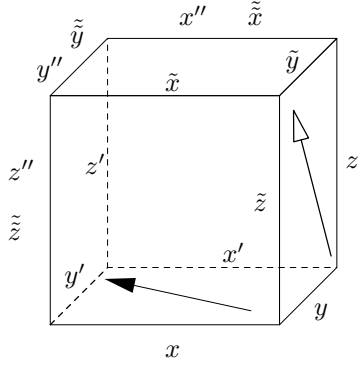
\begin{center} {\em Fig.2 Representation of the Yang--Baxter property}
\end{center}
The first map corresponds to the composition of the down, back, left
faces, while the second one to the right, front and upper faces. All
the parallel edges to the $x$ (resp. $y,z$) axis carry the parameter
$\alpha$ (resp. $\beta$, $\gamma$). So Eq.(\ref{YBprop}) assures
that $x''= \tilde{\tilde{x}},  \ y''=\tilde{\tilde{y}}, \
z''=\tilde{\tilde{z}}$.

The proof of the following proposition based on the associativity
property of matrix multiplication appeares essentially in
\cite{VES03}.

\begin{prop} \label{anagea}
Let $u=(x,y)$, $v=v(x,y) $ and $A(x;\alpha)$ a matrix depending
on a point $x$, a parameter $\alpha$ and a spectral parameter $\zeta$, such that
$A(u;\alpha)A(v;\beta)=A(y;\beta)A(x;\alpha)$. \\
If the equation
\begin{equation}
A( \hat{x}; \alpha )A( \hat{y}; \beta )A(\hat{z}; \gamma )=
A(x; \alpha )A(y; \beta )A(z; \gamma ) \label{xyz}
\end{equation}
implies that $\hat{x}=x, \ \hat{y}=y$ and $\hat{z}=z$,
then the map
$R_{\alpha,\beta}(x,y)=(u,v)$
is a parametric Yang-Baxter map with Lax matrix $A(x;\alpha)$.
\end{prop}

\begin{proof}
Let $A$ be a matrix with the above properties. The cubic
representation of Fig.2 at the down, back and left faces give
$A(y;\beta)A(x;\alpha)=A(x';\alpha)A(y';\beta)$, so
$A(z;\gamma)A(y;\beta)A(x;\alpha)=(A(z;\gamma)A(x';\alpha))A(y';\beta)
=A(x'';\alpha)(A(z';\gamma)A(y';\beta))
=A(x'';\alpha)A(y''\beta)A(z'';\gamma)$.
Similarly from the right, front and upper faces we get
$A(z;\gamma)A(y;\beta)A(x;\alpha)=A(\tilde{\tilde{x}};\alpha)A(\tilde{\tilde{y}};\beta)A(\tilde{\tilde{z}};\gamma)$.
So we have that
$$A(x'';\alpha)A(y''\beta)A(z'';\gamma)=
A(\tilde{\tilde{x}};\alpha)A(\tilde{\tilde{y}};\beta)A(\tilde{\tilde{z}};\gamma)$$
which implies that $x''= \tilde{\tilde{x}},  \ y''=\tilde{\tilde{y}}, \ z''=\tilde{\tilde{z}}$.
i.e. the Yang-Baxter property (\ref{YBprop}).
\end{proof}

\section{Yang--Baxter maps from matrix re-factorization} \label{sec2}

Our aim is to find the YB maps corresponding to those $2 \times 2$ Lax
matrices that are first--degree matrix polynomials with respect
to the spectral parameter. So we consider the set $\mathcal{L}$ of
$2\times 2$ polynomial matrices of the form $L(\zeta)=A-\zeta B$,
$\zeta\in\mathbb{C}$. Let us fix  a constant matrix $B$ in
$GL_{2}(\mathbb{C})$. We denote by $i_B$ the immersion
$i_B:GL_{2}(\mathbb{C})\rightarrow \mathcal{L}$ with $i_B(A)=A-\zeta
B$ and by $p_{A}$ the polynomial
$$p_{A}(\zeta ):=\det(A-\zeta B)\equiv f_{2}(A) \zeta^{2}-f_{1}(A) \zeta + f_{0}(A)$$
where the scalar functions $f_{i}$, $i=0,1,2$ are given by
$$f_{2}(A)=\det B, \ \ f_{1}(A)=\det BTr(AB^{-1}), \ \ f_{0}(A)= \det A .$$
We also define the functions $P_{i}:GL_{2}(\mathbb{C})\times
GL_{2}(\mathbb{C})\rightarrow GL_{2}(\mathbb{C})$ for $i=1,2$, with
\begin{eqnarray} \label{p1p2}
 \ \ \ \ \ \ P_{1}(X,Y) &=& f_{2}(X)(YB+BX)-f_{1}(X)B^{2}, \\
 \ \ \ \ \ \ P_{2}(X,Y) &=& f_{2}(X)YX-f_{0}(X)B^{2}.
\end{eqnarray}
Let $X$, $Y$ be generic elements of $GL_{2}(\mathbb{C})$. We want to
find $U=U(X,Y), \ V=V(X,Y)$, such that the equation
\begin{equation}
i_{B}(U)i_{B}(V)=i_{B}(Y)i_{B}(X) \label{fact}
\end{equation}
holds for any $\zeta \in \mathbb{C} $. First we notice that this
equation admits the trivial solution $U=Y$, $V=X$. The next
proposition give us a second more interesting solution.

\begin{prop} \label{gen}


Let $X, \ Y \in GL_{2}(\mathbb{C})$ such that $\det P_1(X,Y)\neq 0$.
Then there are unique $U=U(X,Y)$ and $V=V(X,Y)$ in
$GL_{2}(\mathbb{C})$, where
\begin{equation} \label{UV}
U(X,Y)= P_{2}(X,Y)P_{1}(X,Y)^{-1}B, \ \ \ V(X,Y)= B^{-1}(YB+BX-UB),
\end{equation}
that satisfy equation (\ref{fact}) and the constraint $\det(U-Y)\neq
0$ (equivalently $\det(V-X)\neq 0$). The map
$R(X,Y)=(U(X,Y),V(X,Y))$ is a (non--parametric) quadrirational
Yang-Baxter map such that $f_{i}(U)=f_{i}(X)$ and
$f_{i}(V)=f_{i}(Y)$ for $i=0,1,2$.
\end{prop}

\begin{proof}
Equation (\ref{fact}) is equivalent with the system:
\begin{equation}
UV=YX,\ \ UB+BV=YB+BX \ .  \label{system}
\end{equation}
If we write the first equation as $UB^{-1}BV=YX$, replace $BV$
from the second one and after some simple algebra, we have that
\begin{equation}
UB^{-1}(Y-U) =(Y-U)XB^{-1} \, , \, (X-V)B^{-1}V = B^{-1}Y(X-V) \ . \label{system2}
\end{equation}
 These two relations show
that if there exist a solution of (\ref{fact}) with $\det(U-Y)\neq
0$ (equivalently $\det(V-X)\neq 0$), then the matrices $UB^{-1}, \
B^{-1}V$ must be similar to the matrices $XB^{-1}$ and $B^{-1}Y$
respectively. So $p_{U}(\zeta )=p_{X}(\zeta )$ and $p_{V}(\zeta
)=p_{Y}(\zeta )$.

Suppose that $U$, $V$ is a solution of equation (\ref{fact}) with
$\det(U-Y)\neq 0$. Cayley--Hamilton theorem states that
$p_{U}(UB^{-1})=0$. Since $p_{U}(\zeta )=p_{X}(\zeta )$, we get that
$p_{X}(UB^{-1})=0$ i.e.
\begin{equation} \label{ch1}
f_{2}(X)(UB^{-1})^{2}-f_{1}(X)(UB^{-1})=-f_{0}(X)I.
\end{equation}
Also system (\ref{system}) gives :
$(UB^{-1})^{2}B^{2}=UB^{-1}(YB+BX)-YX$. So by solving (\ref{ch1})
with respect to $U$ and substituting (\ref{system}) we obtain
(\ref{UV}). Here we have assumed that $\det P_{1}(X,Y)\neq 0 $ for
the generic elements $X$, $Y$ (see remark \ref{remark1}).

So far we have proved that if  a solution exists with
$\det(U-Y)\neq 0$, then it will be unique and will have the form
(\ref{UV}). We still have to check that these $U, \ V$ satisfy
system (\ref{system}). The second equation of the system is
obviously satisfied. Now (\ref{UV}) implies that
$UB^{-1}(f_{2}(X)(YB+BX)-f_{1}(X)B^{2})=(f_{2}(X)YX-f_{0}(X)B^{2})$
and $YB+BX=UB+BV$, so
\begin{equation}
(f_{2}(X)(UB^{-1})^{2}-f_{1}(X)UB^{-1}+f_{0}(X)I)B^{2}=f_{2}(X)(YX-UV). \label{ant}
\end{equation}
Moreover we can write
\begin{eqnarray*}
P_{2}(X,Y)=P_{2}(X,Y)+P_{1}(X,Y)B^{-1}X-P_{1}(X,Y)B^{-1}X = \\
P_{1}(X,Y)B^{-1}X-B^{2}(f_{2}(X)(B^{-1}X)^{2}-f_{1}(X)B^{-1}X+f_{0}(X)I)
\end{eqnarray*}
so $P_{2}(X,Y)=P_{1}(X,Y)B^{-1}X.$ (Here we used again
Cayley--Hamilton theorem). Thus we have the following equivalent
expression for U:
\begin{equation}
UB^{-1}=P_{1}(X,Y)B^{-1}XB^{-1}(P_{1}(X,Y)B^{-1})^{-1} \label{U2}
\end{equation}
Which means that $f_{i}(X)=f_{i}(U)$ for $i=0,1,2$. So from
(\ref{ant}) and the Cayley--Hamilton theorem we finally derive that
$UV=YX$.

We define now  the map 
\begin{equation}R(X,Y)=(U,V)\label{RXY}
\end{equation}
 where $U, \ V$ is the 
solution of (\ref{system}) given by (\ref{UV}).
The quadrirationality of this map is already proven since (\ref{system2}) yields $V, \
X$, in terms of $Y, \ U$:
$$VB^{-1}=(U-Y)^{-1}YB^{-1}(U-Y), \ \
XB^{-1}=(Y-U)^{-1}UB^{-1}(Y-U).$$
from which we had already that $f_{i}(X)=f_{i}(U)$ and $f_{i}(Y)=f_{i}(V)$. 
\end{proof}

We will refer to (\ref{RXY}) as the {\em general Yang-Baxter map associated
with the matrix $B$} and denote it by $\mathcal{R}_{B}$.

\begin{rem} \label{remark1}
The functions $U(X,Y)=U$, $V(X,Y)=V$ of (\ref{UV}) are rational and
of course not defined everywhere on $\mathbb{C}^4 \times
\mathbb{C}^4$ but just in an open and dense domain $I \subset
\mathbb{C}^4 \times \mathbb{C}^4$ defined by the restriction $\det
P_{1}(X,Y)\neq 0$. Proposition \ref{gen} holds in this domain.
\end{rem}
\section{Poisson Structure and Reduction} \label{sec3}
We equip the manifold $\mathcal{L}$ with the Sklyanin bracket
\cite{SKL83}. We will show how we can reduce the general Yang-Baxter
Map to Poisson submanifolds in order to obtain Poisson parametric
Yang-Baxter Maps on these submanifolds. This reduction is possible
due to the fact that the Casimir functions of this structure are
exactly the invariant functions $f_{i}$ that defined in the previous
section. If the corresponding level set is a symplectic submanifold
then the reduced YB map is symplectic.
\subsection{Poisson Structure on $\mathcal{L}$} \label{section31}

The Sklyanin bracket between the variables of a matrix polynomial
$L(\zeta)$ of any degree is given by the formula :
\begin{equation}
\{L(\zeta ) \ \overset{\otimes }{,} \ L(\eta)\}=[\frac{r}{\zeta
-\eta},L(\zeta )\otimes L(\eta)]  \label{sklyanin}
\end{equation}
Here $r$ denotes the permutation matrix: $r(x\otimes y) = y\otimes
x$. The restriction of this bracket on the submanifold
$\mathcal{L}$, of functions of the form $L(\zeta )=A-\zeta B$, with
\\ $A=
\begin{pmatrix}
a _{1} & a_2 \\
a_3 & a_{4}
\end{pmatrix}$ and
$B=
\begin{pmatrix}
b_1 & b_2 \\
b_3 & b_4
\end{pmatrix},
$ is given by the Poisson structure anti--symmetric matrix :
\begin{equation}
J_{B}(A)=
\begin{pmatrix}
0 & -a_{2}b_{1}+a_{1}b_{2} & a_{3}b_{1}-a_{1}b_{3} &
a_{3}b_{2}-a_{2}b_{3} \\
* & 0 & a_{4}b_{1}-a_{1}b_{4} &
a_{4}b_{2}-a_{2}b_{4} \\
* & * & 0 &
-a_{4}b_{3}+a_{3}b_{4} \\
* & * & *& 0
\end{pmatrix}
\label{strmatrix}
\end{equation}
where $J_{B}(A)_{ij}$ denotes the bracket $\{ a_{i}- \zeta
b_{i},a_{j}- \zeta b_{j} \}$, for $i,j=1,...,4$.

First we notice that matrix $B$ belongs to the center of this
Poisson algebra (so $J_{B}(A)_{ij}$ is just $\{ a_{i},a_{j} \}$). As
in \cite{RV} we restrict to the level set for $B=Constant$ and
denote this Poisson submanifold by $\mathcal{L}_{B}$. The Casimir
functions for the Poisson structure (\ref{strmatrix}) on
$\mathcal{L}_{B}$ are:
$$f_{0}(A)=\det A,\ \
f_{1}(A)=a_{1}b_{4}+a_{4}b_{1}-a_{3}b_{2}-a_{2}b_{3} $$ These are
the coefficients of $\det(A-\zeta B)$ and agree with $f_{0}$,
$f_{1}$ defined in section \ref{sec2}. By this notation the general
YB map that we constructed in the previous section is a
non--parametric YB map $\mathcal{R}_{B}:\mathcal{L}_{B}\times
\mathcal{L}_{B}\rightarrow \mathcal{L}_{B}\times \mathcal{L}_{B}$.

We can extend the Poisson bracket of $\mathcal{L}_{B}$ to the Cartesian product
$\mathcal{L}_{B}\times \mathcal{L}_{B}$ as follows :
\begin{equation}
\{x_{i},x_{j}\}=J_{B}(X)_{ij},\ \{y_{i},y_{j}\}=J_{B}(Y)_{ij},\
\{x_{i},y_{j}\}=0,  \label{bracket}
\end{equation}
for any $( X-\zeta B,\ Y-\zeta B )\in \mathcal{L}_{B}\times \mathcal{L}_{B}$ where
$x_{i},\ x_{j},\ y_{i},\ y_{j}$ for $i=1,...,4$ are the elements of the
matrices $X,\ Y$ respectively and $J_{B}$ the matrix of the Poisson structure
(\ref{strmatrix}).

\begin{prop}
The general YB map
$\mathcal{R}_{B}:\mathcal{L}_{B}\times \mathcal{L}_{B}\rightarrow  \mathcal{L}_{B}\times \mathcal{L}_{B}$
is a Poisson map.
\end{prop}
\begin{proof}
A detailed computation shows that the Poisson bracket between the
entries of $U,V$ defined by (\ref{UV}) is:

$$\{u_{i},u_{j}\}=J_{B}(U)_{ij},\ \{v_{i},v_{j}\}=J_{B}(V)_{ij},\
\{u_{i},v_{j}\}=0,$$
for $i=1,...,4$.
\end{proof}
\begin{rem}
Let $\mathfrak{g}$ be the four dimensional four parametric Lie algebra, with basis
$\{e_1,e_2,e_3,e_4\}$, defined by
\begin{eqnarray*}
\lbrack e_{1},e_{2}] &=&b_{2}e_{1}-b_{1}e_{2},\ \  [e_{1},e_{3}]=
-b_{3}e_{1}+b_{1}e_{3},\ \ [e_{1},e_{4}]=-b_{3}e_{2}+b_{2}e_{3},
\\
\lbrack e_{2},e_{3}] &=&b_{1}e_{4}-b_{4}e_{1},\ \
[e_{2},e_{4}]=-b_{4}e_{2}+b_{2}e_{4},\ \
[e_{3},e_{4}]=b_{4}e_{3}-b_{3}e_{4},
\end{eqnarray*}
where $b_i$ for $i=1,...,4$ are free parameters. Then the Poisson structure (\ref{strmatrix})
on $\mathcal{L}_B$ coincides with the corresponding Lie--Poisson structure
on the dual $\mathfrak{g}^*$:
$$\{F,G\}_{L-P}(x)=\left\langle x,[d_{x}F,d_{x}G]\right\rangle ,~~x\in
\mathfrak{g}^{\ast },\ F, \ G\in C^{\infty }(\mathfrak{g}^{\ast }).$$
\end{rem}

\subsection{Parametric YB maps and Lax Matrices} \label{section32}

Let $A-\zeta B$ be a generic element of $\mathcal{L}_{B}$ and
$a_{ij}$ an element of $A$ with $\frac{\partial f_{1}}{\partial a_{ij}}\neq 0$.
If we set  $f_{0}(A)=c$, then there exist a
function $F_{0}$ such that $F_{0}(a_{1},a_{2},a_{3},c)=a_{ij}$,
where $a_{1},a_{2},a_{3}$ here and below denote the remaining three entries of $A$.
We denote by $L'_{0}(a_{1},a_{2},a_{3};c )$ the matrix
that is derived by replacing the $a_{ij}$ element of $A$ by  $F_{0}(a_{1},a_{2},a_{3},c )$, and
by $L_{0}(a_{1},a_{2},a_{3};c )$ the matrix
$i_B(L'_{0}(a_{1},a_{2},a_{3};c ))$.
We also define the projection $p_{ij}:GL_2(\mathbb{C})\rightarrow \mathbb{C}^3$ to the
elements of a matrix except of the $ij$ element and the function
$P:GL_2(\mathbb{C})\times GL_2(\mathbb{C})\rightarrow \mathbb{C}^3\times \mathbb{C}^3$
with $P(X,Y)=(p_{ij}(X),p_{ij}(Y))$.

In a similar way if $a_{ij}$ is an element of $A$ such that
$\frac{\partial f_{1}}{\partial a_{ij}}\neq 0$, we define the matrix
$ L'_{1}(a_{1},a_{2},a_{3};c)$ by setting $\ f_{1}(A)=c$, the matrix
$L_{1}(a_{1},a_{2},a_{3};c)=L'_{1}(a_{1},a_{2},a_{3};c)-\zeta B$ and
the corresponding projection.

Let also\ $a_{ij},\ a_{kl}$ be two elements of $A$ such that
$\det \ [ \frac{\partial (f_{0},f_{1})}{\partial (a_{ij},a_{kl})}  ]\neq 0$.
If we set
$f_{0}(A)=c_{1}$ and $f_{1}(A)=c_{2}$, there exist two functions $G_{0}, \ G_{1}$ such that
$G_{0}(a_{1},a_{2},c_{1} ,c_{2} )=a_{ij}$ and
$G_{1}(a_{1},a_{2},c_{1} ,c_{2} )=a_{kl}$.
We denote by
$L'(a_{1},a_{2};c_{1} ,c_{2} )$ the matrix that is obtained by replacing
the $a_{ij}$ and the $a_{kl}$ elements of $A$ by
$G_{0}(a_{1},a_{2},c_{1} ,c_{2} )$ and $G_{1}(a_{1},a_{2},c_{1},c_{2} )$
respectively and
$L(a_{1},a_{2};c_{1} ,c_{2} )=i_{B}(L'(a_{1},a_{2};c_{1} ,c_{2} ))$.
We also define the projection $q_{ij,kl}:GL(2)\rightarrow \mathbb{C}^2$ to the elements of a matrix except the
$ij$ and $kl$ elements and the function
$Q:GL(2)\times GL(2)\rightarrow \mathbb{C}^2\times \mathbb{C}^2$ with
$Q(X,Y)=(q_{ij,kl}(X),q_{ij,kl}(Y))$.

Since $f_{0},\ f_{1}$ are Casimir functions the sets
\begin{eqnarray*}
\mathcal{P}_{B_{0}}(c) &=&\{L_{0}(a_{1},a_{2},a_{3};c )
\in \mathcal{L}_{B}\ /\ f_{0}(L_{0}(a_{1},a_{2},a_{3};c ))=c \}, \\
\mathcal{P}_{B_{1}}(c ) &=&\{L_{1}(a_{1},a_{2},a_{3};c )\in \mathcal{L}%
_{B}\ /\ f_{1}(L_{1}(a_{1},a_{2},a_{3};c ))=c \}
\end{eqnarray*}

are Poisson submanifolds of $\mathcal{L}_{B}$ and the sets
$$
\mathcal{S}_{B}(c_{1} ,c_{2} ) =
\{ L(a_{1},a_{2};c_{1} ,c_{2} )\in \mathcal{L}_{B}\ / \
f_{0}(L(a_{1},a_{2};c_{1} ,c_{2} )=c_{1},
f_{1}(L(a_{1},a_{2};c_{1} ,c_{2} )=c_{2} \}
$$
are two dimensional symplectic submanifolds of $\mathcal{L}_{B}$
with the reductive symplectic form from (\ref{strmatrix}).

Now we return to the general YB map
$\mathcal{R}_{B}(X,Y)=(U(X,Y),V(X,Y))$, where $U(X,Y)$ and $V(X,Y)$ are defined by
(\ref{UV}).
We can reduce this map to parametric YB maps on the
sets $\mathcal{P}_{B_{0}}\times \mathcal{P}_{B_{0}}$,
$\mathcal{P}_{B_{1}}\times\mathcal{P}_{B_{1}}$ and
$\mathcal{S}_{B}\times \mathcal{S}_{B}$.
\begin{prop} \label{param}
The maps $R^{0} _{\alpha,\beta}$, $R^{1} _{\alpha,\beta}$ defined by
\begin{eqnarray*}
R^{0} _{\alpha,\beta}(x,y)&=&P\circ \mathcal{R}_{B}(L'_{0}(x_{1},x_{2},x_{3};\alpha ),
L'_{0}(y_{1},y_{2},y_{3};\beta ) ) \\
R^{1} _{\alpha,\beta}(x,y)&=&P\circ \mathcal{R}_{B}(L'_{1}(x_{1},x_{2},x_{3};\alpha ),
L'_{1}(y_{1},y_{2},y_{3};\beta ) ),
\end{eqnarray*}
where $x=(x_1,x_2,x_3), \ y=(y_1,y_2,y_3)$, are quadrirational
Poisson parametric Yang-Baxter maps on $\mathbb{C}^3\times
\mathbb{C}^3$ with parameters $\alpha$, $\beta$ and Lax matrices
$L_{0}(x_{1},x_{2},x_{3};\alpha)$, $L_{1}(x_{1},x_{2},x_{3};\alpha)$
respectively.
The map
\begin{equation} \label{symYB}
R _{\alpha,\beta}((x_1,x_2),(y_1,y_2))=Q \circ \mathcal{R}_{B}(L'(x_{1},x_{2};\alpha_{1},\alpha_{2}),
L'(y_{1},y_{2};\beta_{1},\beta_{2})),
\end{equation}
is a quadrirational symplectic
parametric Yang-Baxter map on $\mathbb{C}^2\times \mathbb{C}^2$ with
vector parameters $\alpha=(\alpha_{1}, \alpha_{2})$,
$\beta=(\beta_{1},\beta_{2})$ and Lax matrix
$L(x_{1},x_{2};\alpha_{1},\alpha_{2})$.
\end{prop}

\begin{proof}
The construction of the matrices $L'_{0}(x_{1},x_{2},x_{3};\alpha )$
and $L'_{0}(y_{1},y_{2},y_{3};\beta )$ implies that there are $X, \
Y \in GL_2(\mathbb{C})$ such that $L'_{0}(x_{1},x_{2},x_{3};\alpha
)=X$, $L'_{0}(y_{1},y_{2},y_{3};\beta )=Y$ and $f_{0}(X)=\alpha$,
$f_{0}(Y)=\beta$. Then
$$\mathcal{R}_{B}(L'_{0}(x_{1},x_{2},x_{3};\alpha ),
L'_{0}(y_{1},y_{2},y_{3};\beta )) = \mathcal{R}_{B}(X,Y)=(U,V),$$
where
$U=U(X,Y)$, $V=V(X,Y)$ are defined
by (\ref{UV}).  This is a Poisson YB map such that
$f_{0}(U)=\alpha$, $f_{0}(V)=\beta$.
So $U=L'_{0}(u_{1},u_{2},u_{3};\alpha )$,
$V=L'_{0}(v_{1},v_{2},v_{3};\beta )$.
The projection $P(U,V)$ give us
the corresponding elements $u=(u_{1},u_{2},u_{3}),v=(v_{1},v_{2},v_{3})$.
The Yang Baxter property of this map, as well as the quadrirationality,
are immediately derived from the YB property and the quadrirationality of $\mathcal{R}_{B}$.

Since $i_{B}(U)i_{B}(V)=i_{B}(Y)i_{B}(X)$
we'll have that
$$(L'_{0}(u;\alpha )-\zeta B)(L'_{0}(v;\beta )-\zeta B)=
(L'_{0}(y;\beta )-\zeta B)(L'_{0}(x;\alpha )-\zeta B)$$ or
equivalently $L_{0}(u_{1},u_{2},u_{3};\alpha
)L_{0}(v_{1},v_{2},v_{3};\beta )= L_{0}(y_{1},y_{2},y_{3};\beta
)L_{0}(x_{1},x_{2},x_{3};\alpha )$ which means that
$L_{0}(x_{1},x_{2},x_{3};\alpha )$ is a Lax matrix for this YB map.
(Alternatively we can prove the YB property from proposition
\ref{anagea} by showing that the equation $L_{0}(x';\alpha
)L_{0}(y';\beta )L_{0}(z';\gamma )=L_{0}(x;\alpha )L_{0}(y;\beta
)L_{0}(z;\gamma )$ implies $x'=x,\ y'=y$ and $z'=z$).

The proof
for the other maps is similar.
\end{proof}

In correspondence with remark \ref{remark1},
when we refer to YB maps on $\mathbb{C}^3 \times \mathbb{C}^3$
we mean on the open and dense domain of it that are defined
(respectively for $\mathbb{C}^2 \times \mathbb{C}^2$).

Now since equation (\ref{fact}) has unique solution when
$f_{0}(U)=f_{0}(X)=\alpha_{1}$, $f_{1}(U)=f_{1}(X)=\alpha_{2}$,
$f_{0}(V)=f_{0}(Y)=\beta_{1}$ and $f_{1}(V)=f_{1}(Y)=\beta_{2}$,
according to the above, the next corollary holds.
\begin{corollary} \label{cor}
The equation
\begin{equation}
L(u_{1},u_{2};\alpha_{1},\alpha_{2})L(v_{1},v_{2};\beta_{1},\beta_{2})=
L(y_{1},y_{2};\beta_{1},\beta_{2})L(x_{1},x_{2};\alpha_{1},\alpha_{2})
\label{lax}
\end{equation}
is uniquely solvable with respect to $u_{1},u_{2},v_{1},v_{2}$. The
mapping $R_{\alpha,\beta}(x,y)=(u,v)$ with vector parameters
$\alpha=(\alpha_{1},\alpha_{2}), \beta=(\beta_{1},\beta_{2})$,
variables $x=(x_{1},x_{2}),\ y=(y_{1},y_{2})$ and $u=(u_{1},u_{2}),\
v=(v_{1},v_{2})$ the unique solution of (\ref{lax}), is the
symplectic parametric Yang-Baxter map (\ref{symYB}) of Prop.
\ref{param}.
\end{corollary}

\begin{rem} \label{symMaps}
By setting $\alpha_{1}=\beta_{1}=k$, the symplectic YB map of \ref{param}
yields the symplectic parametric YB map
$ R_{\alpha_{2} \beta_{2}}$ with Lax matrix $L(x_{1},x_{2};\alpha_{2} )$ $:=L(x_{1},x_{2};k,\alpha_{2} )$.
Here $k$ is not a dynamical parameter as $\alpha_{2}$ but just a free parameter. We have analogous result by setting
$\alpha_{2}=\beta_{2}$.
If we set $\alpha_{1}=\beta_{1}$ and $\alpha_{2}=\beta_{2}$ then we derive the trivial
solution $U=Y$, $V=X$. That holds because this is the only solution of Eq.(\ref{fact}) with
$f_{0}(U)=f_{0}(Y)$ and $f_{1}(U)=f_{1}(Y)$.
\end{rem}

\subsection{Classification by normal forms} \label{33}

Equation (\ref{fact}) is invariant under conjugation i.e:
$$P^{-1}(U-\zeta B)PP^{-1}(V-\zeta B)P=P^{-1}(Y-\zeta B)PP^{-1}(X-\zeta B)P$$
The same holds also for the Casimir functions $f_{0}$, $f_{1}$,
since $f_{0}(A)$, $f_{1}(A)$ are the coefficients of $\det(A-\zeta
B)$. This means that we can restrict our attention to the next
Jordan normal forms for the matrix $B$:
$$
i)
\begin{pmatrix}
\lambda _{1} & 0 \\
0 & \lambda _{2}
\end{pmatrix}
,\ \  ii)
\begin{pmatrix}
\lambda  & 0 \\
0 & \lambda
\end{pmatrix}
,\ \ iii)
\begin{pmatrix}
\lambda  & 1 \\
0 & \lambda
\end{pmatrix}.
$$

More precisely let $B_0$ be one of these normal forms and
$$
\tilde{R}_{\alpha,\beta}(x,y)=Q \circ
\mathcal{R}_{B_0}(L_0'(x_{1},x_{2};\alpha_{1},\alpha_{2}),
L_0'(y_{1},y_{2};\beta_{1},\beta_{2})),
$$
the symplectic YB map of Prop \ref{param} with corresponding Lax
matrix $L_0(a_{1},a_{2};c_{1},c_{2})$. Let also $B$ a similar matrix
with $B_0$, $B=P B_0 P^{-1}$. Then the map
\begin{equation} \label{jord}
R_{\alpha,\beta}(x,y)=Q(P [\mathcal{R}_{B_0}( P^{-1}
L_0'(x_{1},x_{2};\alpha_{1},\alpha_{2})P,
P^{-1}L_0'(y_{1},y_{2};\beta_{1},\beta_{2})P)] P^{-1}),
\end{equation}
is a symplectic YB map with Lax matrix
$L(a_{1},a_{2};c_{1},c_{2})=PL_0(a_{1},a_{2};c_{1},c_{2})P^{-1}$ and
is exactly the unique solution of (\ref{lax}) with respect to
$u=(u_1,u_2), \ v=(v_1,v_2)$. (The first and the last multiplication
by $P$ and $P^{-1}$ respectively of (\ref{jord}) is done at each
factor of the cartesian product
$\mathcal{S}_{B}(\alpha_{1},\alpha_{2})\times
\mathcal{S}_{B}(\beta_{1},\beta_{2})$ ). Similar results hold for
the Poisson maps $R^{0} _{\alpha,\beta}$ and $R^{1} _{\alpha,\beta}$
of Prop. \ref{param}.

If we are interested in real Lax matrices we have to include
also the case where
$B_0= \begin{pmatrix}
\lambda _{1} & - \lambda _{2} \\
 \lambda _{2} & \lambda _{1}
\end{pmatrix}.$

\vspace{0.5cm}

\begin{expl}
We are going to apply the above results for $B=I$. In this case the non--zero Poisson
brackets of $\mathcal{L}_{I}$ are given by the relations:
\begin{eqnarray*}
 \{a_{11},a_{12} \} =-a_{12}, \ \{a_{11},a_{21} \}
=a_{21}, \ \{a_{12},a_{21} \} =a_{22}-a_{11}, \\
 \{a_{12},a_{22} \}
=-a_{12}, \ \{a_{21},a_{22} \} =a_{21}.
\end{eqnarray*}
 The Casimir functions are
$f_{0}(A)=\det A,\ f_{1}(A)=a_{11}+a_{22}$. If we set
$f_{0}(A)=c_{1} ,\ f_{1}(A)=c_{2}$, and solve with respect to
$a_{11}, \ a_{12}$, for $a_{12} \neq 0$, we come up to the Lax
matrix
$$
L(a_{1},a_{2};c_{1} ,c_{2} )=
\begin{pmatrix}
a_{1}-\zeta  & a_{2} \\
\frac{a_{1}(c_{2} -a_{1})-c_{1} }{a_{2}} & c_{2} -a_{1}-\zeta
\end{pmatrix},
$$
where $a_{1},\ a_{2}$ here denotes $a_{11}, \ a_{12}$ respectively.
According to \ref{cor} the unique solution of the equation
$L(u_{1},u_{2};\alpha_{1},\alpha_{2})L(v_{1},v_{2};\beta_{1},\beta_{2})=
L(y_{1},y_{2};\beta_{1},\beta_{2})L(x_{1},x_{2};\alpha_{1},\alpha_{2})$
for any $\zeta \in\mathbb{C}$, gives the parametric YB map on $\mathbb{C}^2\times \mathbb{C}^2$:
$$
R_{\alpha,\beta}((x_{1},x_{2}),(y_{1},y_{2}))=((u_1,u_2),(v_1,v_2))=
Q \circ \mathcal{R}_{I}(L'(x_{1},x_{2};\alpha_{1},\alpha_{2}),
L'(y_{1},y_{2};\beta_{1},\beta_{2}))
$$
where $L'(x_{1},x_{2};\alpha_{1},\alpha_{2})=L(x_{1},x_{2};\alpha_{1},\alpha_{2})+\zeta I$.
This map is symplectic with respect to the reduced symplectic structure defined by :
$$\{ x_1,x_2 \}=-x_2, \ \{ y_1,y_2\}=-y_2, \ \{x_i,y_j \}=0 ,\ \ i=1, 2,$$ on the corresponding
symplectic leave
$=\{(L(x_{1},x_{2};\alpha_{1},\alpha_{2}), L(y_{1},y_{2};
\beta_{1},\beta_{2})) \ /  \ x_1,  y_1 \in  \mathbb{C}, \ x_2,  y_2 \in  \mathbb{C}^* \}$ of
$\mathcal{L}_I \times \mathcal{L}_I$.

We can restrict matrix $A$ to $SL_{2}(\mathbb{C})$ by setting
$f_{0}(A)=1, \ f_{1}(A)=c$. In this case the corresponding Lax matrix will be
$M(a_{1},a_{2};c )=L(a_{1},a_{2};1,c)$.
Now the unique solution of the equation
$M(u_{1},u_{2};\alpha)M(v_{1},v_{2};\beta)=M(y_{1},y_{2};\beta)M(x_{1},x_{2};\alpha)$
gives the parametric YB map:
$$R'_{\alpha,\beta}((x_{1},x_{2}),(y_{1},y_{2}))=Q \circ \mathcal{R}_{I}(L'(x_{1},x_{2};1,\alpha),
L'(y_{1},y_{2};1,\beta)).$$
\end{expl}

\section{Non--quadrirational Yang--Baxter maps}
Non--quadrirational YB maps arise when
the constant matrix $B$ of $\mathcal{L_{B}}$ is non--invertible. In
this section we show how, in some cases, we can obtain non--quadrirational
parametric YB maps as limits of the quadrirational maps of
the previous section.

We consider invertible constant matrices $B=B(\varepsilon )$
depending on a parameter $\varepsilon $ such that
$\underset{\varepsilon \rightarrow 0}{\lim} \det{B}=0$, and
construct the quadrirational symplectic YB map $R(\varepsilon )$
with the corresponding Lax matrix $L(\varepsilon )$ of Prop.
\ref{param}. By taking the limit of $R(\varepsilon )$ for
$\varepsilon \rightarrow 0$ we derive a rational non--quadrirational YB map
on $\mathbb{C}^2\times \mathbb{C}^2$. This map is symplectic with
the sumplectic form that is induced by taking the limit of the
stucture matrix $J_{B(\varepsilon )}(L(\varepsilon)+\zeta
B(\varepsilon ))$. We restrict our analysis to the Jordan normal
forms.
\subsection{The Adler-Yamilov map} \label{adl}
Consider
$B=
\begin{pmatrix}
1 & 0 \\
0 & \varepsilon
\end{pmatrix}
$.
The Casimir functions on $\mathcal{L_{B}}$ will be
$$
f_{0}(A)=\det A \ ,\ \
f_{1}(A)=a_{11}\varepsilon +a_{22}
$$
We set $f_{0}(A)=c$, $f_{1}(A)=1$ and solve with respect to $a_{11}$, $a_{22}$:
$$
a_{22}=1-\varepsilon a_{11},\ \
a_{11}=\frac{1-\sqrt{1-4\varepsilon (c+a_{12}a_{21})}}{2\varepsilon}
$$
Now we can construct the Lax matrix
\begin{equation} \label{deg1}
L(a_{1},a_{2};c ,1 )=
\begin{pmatrix}
\frac{1-\sqrt{1-4\varepsilon (c+a_{1}a_{2})}}{2\varepsilon }-\zeta
& a_{1} \\
a_{2} & \frac{1}{2}(1+\sqrt{1-4\varepsilon (c+a_{1}a_{2})})-\varepsilon \zeta
\end{pmatrix}
\end{equation}
where $a_{1}=a_{12}, \ a_{2}=a_{21}$. According to the previous section, the unique solution of the equation
$$
L(u_{1},u_{2};\alpha,1)L(v_{1},v_{2};\beta,1)=
L(y_{1},y_{2};\beta,1)L(x_{1},x_{2};\alpha,1)
$$
will be $u_{1}=u_{12}$, $u_{2}=u_{21}$, $v_{1}=v_{12}$, $v_{2}=v_{21}$ where $u_{ij}, \ v_{ij}$
are the corresponding elements of the matrices:
\begin{eqnarray*}
U = (f_{2}(X)YX-f_{0}(X)B^{2})(f_{2}(X)(YB+BX)-f_{1}(X)B^{2})^{-1}B \\
V = B^{-1}(YB+BX-UB)
\end{eqnarray*}
by setting $X=L'(x_{1},x_{2};\alpha,1)$ and
$Y=L'(y_{1},y_{2};\beta,1)$ (from definition $L'=L+ \zeta B $). Here
$f_{2}(X)=\det B=\varepsilon $, $f_1(X)=1$ and $f_0(X)=\alpha$. This
solution gives the quadrirational parametric YB map:
$${R}_{\alpha \beta }((x_{1},x_{2}),(y_{1},y_{2}))=(({u}_{1},{u}_{2}),
({v}_{1},{v}_{2})).$$

Now if we take the limit of $u_{i},  v_{i}, \ i=1,2$, for $\varepsilon\rightarrow 0$, we derive
\begin{eqnarray*}
\bar{u}_{1} = \lim_{\varepsilon \rightarrow 0}u_{1}=y_{1}-\frac{(a-b)x_{1}}
{1+x_{1}y_{2}} \ , \ \ \ \
\bar{u}_{2}&=&\lim_{\varepsilon \rightarrow 0}u_{2}=y_{2}\ , \\
\bar{v}_{1} =\lim_{\varepsilon \rightarrow 0}v_{1}=x_{1}\ ,\ \ \ \ \ \ \ \ \ \ \ \ \ \ \ \ \ \ \ \ \
\bar{v}_{2}&=&\lim_{\varepsilon \rightarrow 0}
v_{2}=x_{2}+\frac{(a-b)y_{2}}{1+x_{1}y_{2}}
\end{eqnarray*}
and the parametric YB map
$\bar{R}_{\alpha \beta }((x_{1},x_{2}),(y_{1},y_{2}))=((\bar{u}_{1},\bar{u}%
_{2}),(\bar{v}_{1},\bar{v}_{2}))$.
The latter is a map related to the Nonlinear $Schr\ddot{o}dinger$
systems \cite{AY94,PT07}.

The induced symplectic structure is derived from (\ref{bracket}) by taking the limit for $\varepsilon \rightarrow 0$ of
$J_B(L'(x_{1},x_{2};\alpha,1))$ and $J_B(L'(y_{1},y_{2};\beta,1))$ :
$\{x_1,x_2\}=1, \ \{y_1,y_2\}=1, \ \{x_i,y_j\}=0$
i.e. the canonical symplectic form. The YB map $\bar{R}_{\alpha \beta }$ is symplectic with respect to this form.
\subsection{A lift of the KdV quadgraph equation} \label{KdV}
Now we consider
$B=
\begin{pmatrix}
\varepsilon  & 1 \\
0 & \varepsilon
\end{pmatrix}
$.
In this case the Casimir functions on $\mathcal{L_{B}}$ will be
$$
f_{0}(A)=\det A \ ,\ \
f_{1}(A)=\varepsilon (a_{11} +a_{22})-a_{21}.
$$
We set again $f_{0}(A)=c$, $f_{1}(A)=1$ and solve with respect to $a_{21}$, $a_{12}$:
$$
a_{21}=\varepsilon (a_{11} +a_{22})-1,\ \
a_{12}=\frac{a_{11}a_{22}-c}{\varepsilon (a_{11}+a_{22})-1}.
$$
The Lax matrix will be
\begin{equation} \label{deg2}
L(a_{1},a_{2};c ,1 )=
\begin{pmatrix}
a_{1}+\varepsilon\zeta
& \frac{a_{1}a_{2}-c}{\varepsilon (a_{1}+a_{2})-1}+\zeta \\
\varepsilon (a_{1} +a_{2})-1 & a_{2}+\varepsilon\zeta
\end{pmatrix},
\end{equation}
where here $a_{1}$, $a_{2}$ denote $a_{11}$, $a_{22}$ respectively.
As before the unique solution of the equation
$$
L(u_{1},u_{2};\alpha,1)L(v_{1},v_{2};\beta,1)=
L(y_{1},y_{2};\beta,1)L(x_{1},x_{2};\alpha,1)
$$
will be the elements $u_{11}$, $u_{22}$ and $v_{11}$, $v_{22}$ of the matrices
(\ref{UV}) (denoted by $u_{1}, \ u_{2}$ and $v_{1}, \ v_{2}$ respectively),
where
$X=L'(x_{1},x_{2};\alpha,1)$ and
$Y=L'(y_{1},y_{2};\beta,1)$. Here
$f_{2}(X)=\det B=\varepsilon^2 $. So we derive the corresponing quadrirational YB map
${R}_{\alpha \beta }((x_{1},x_{2}),(y_{1},y_{2}))=(({u}_{1},{u}_{2}),
({v}_{1},{v}_{2}))$, with Poisson brackets
\begin{equation} \label{pkdv}
\{x_1,x_4 \}=-1+\varepsilon(x_1+x_4), \ \{y_1,y_4 \}=-1+\varepsilon(y_1+y_4), \ \{x_i,y_j\}=0.
\end{equation}
By taking the limits of $u_{1}, \ u_{2}$ and $v_{1}, \ v_{2}$, for
$\varepsilon\rightarrow 0$, we derive the parametric Yang-Baxter map
$\bar{R}_{\alpha \beta }((x_{1},x_{2}),(y_{1},y_{2}))= (
(\bar{u}_{1},\bar{u}_{2}),(\bar{v}_1,\bar{v}_2) )$ where
\begin{equation} \label{kdvmap}
\bar{u}_1=y_{1}+\frac{\alpha -\beta }{x_{1}+y_{2}}, \ \bar{u}_2=y_{2}, \
\bar{v}_1=x_{1}, \ \bar{v}_2=x_{2}-\frac{\alpha -\beta }{x_{1}+y_{2}}
\end{equation}
This map is symplectic with respect to the induced symplectic form defined by the limit of (\ref{pkdv}):
$$\{x_1,x_2\}=-1, \ \{y_1,y_2\}=-1, \ \{x_i,y_j\}=0.$$

We are going to show that this can be squeezed down to the KdV quadgraph equation.
We perform, first, the following change of variables:
$x_2 \mapsto -x_2$, $y_2 \mapsto -y_2$, $\bar{u}_2 \mapsto -\bar{u}_2$, $\bar{v}_2 \mapsto -\bar{v}_2$ so
$$\bar{u}_1=y_{1}+\frac{\alpha -\beta }{x_{1}-y_{2}}, \ \bar{u}_2=y_{2}, \ \bar{v}_1=x_{1}, \
\bar{v}_2=x_{2}+\frac{\alpha -\beta }{x_{1}-y_{2}}$$
Notice now that if $y_{1}=x_{2}$ then $\bar{u}_{1}=\bar{v}_{2}$ and labeling the variables as
$y_{1}=x_{2}=f$, $\bar{u}_{1}=\bar{v}_{2}=f_{12}$, $\bar{v}_1=x_{1}=f_1$, $y_{2}=\bar{u}_2=f_2$
both first and last  equations reduce to the KdV quadgraph equation
$$(f_{12}-f)(f_1-f_2)=\alpha -\beta .$$
This is the reason why (\ref{kdvmap}) can be thought as a lift of KdV quadgraph equation. Actually this is an instance of 
the fact that all quadgraph equations of the ABS classification in \cite{ABS1}  can be lifted to a 2--field quadgraph 
equation that can be cast into YB map form  \cite{PTprep}.
\begin{rem}
The limits of the Lax matrices (\ref{deg1}) and (\ref{deg2}) for
$\varepsilon \rightarrow 0$ are
$$
L_1(a_{1},a_{2};c)=%
\begin{pmatrix}
a_{1}a_{2}+c-\zeta  & a_{1} \\
a_{2} & 1%
\end{pmatrix}%
,\ \ L_2(a_{1},a_{2};c)=%
\begin{pmatrix}
a_{1} & c-a_{1}a_{2}-\zeta  \\
-1 & a_{2}%
\end{pmatrix}%
$$
respectively. These matrices are Lax matrices of the non--quadrirational YB
maps of \ref{adl} and \ref{KdV} respectively but the equation
$L_2(u_{1},u_{2};\alpha)L_2(v_{1},v_{2};\beta)=L_2(y_{1},y_{2};\beta)L_2(x_{1},x_{2};\alpha)$
is not uniquely solvable with respect to $u_i, v_i$. Therefore we
cannot derive the YB map (\ref{kdvmap}) directly from the Lax matrix
$L_2(a_{1},a_{2};c)$.
\end{rem}

\section{Conclusion} \label{Conclusion}
We saw how through matrix re--factorization and linear algebra considerations
 (namely Caley-Hamilton theorem) one is guided to consider
the Casimirs of the Sklyanin bracket as  the main conditions 
for a non--trivial solution of equation (\ref{fact}). We conjecture 
that a formula analogous to (\ref{UV}) can be found in the case of $n \times n$ matrices ($n>2$).

\appendix
\section{Proof of the YB property of the map given by (\ref{UV})}
We give a detailed proof of the Yang--Baxter property of the map of
\ref{gen}. Let $X,Y,Z$ be generic elements of $GL_{2}(\mathbb{C})$. 
Because of Prop. \ref{anagea}, it
suffices to show that the equation
\begin{equation} \label{xyz1}
(X'-\zeta B)(Y'-\zeta B)(Z'-\zeta B)=(X-\zeta B)(Y-\zeta B)(Z-\zeta B)
\end{equation}
with $\det(X'-\zeta B)=\det(X-\zeta B)$ and $\det(Y'-\zeta
B)=\det(Y-\zeta B)$ implies that $X'=X$, $Y'=Y$ and $Z'=Z$. 

If we set
$
XYZ=K, \ XYB+XBZ+BYZ=L \ and \ XB^{2}+BYB+B^{2}Z=M
$
then from equation (\ref{xyz1}) we derive the system :
$$X^{\prime }Y^{\prime }Z^{\prime }=K, \
X^{\prime }Y^{\prime }B+X^{\prime }BZ^{\prime }+BY^{\prime}Z^
{\prime}=L,  \
 X^{\prime }B^{2}+BY^{\prime }B+B^{2}Z^{\prime }=M$$
Last system implies
\begin{equation} \label{tritis}
(X'B^{-1})^{3}B^{3}-(X'B^{-1})^{2}M+X'B^{-1}L=K.
\end{equation}
Since $\det(X'-\zeta B)=\det(X-\zeta B)$ we have that
\begin{equation} \label{caley3}
f_{2}(X)(X'B^{-1})^{2}-f_{1}(X)(X'B^{-1})+f_{0}(X)I=0
\end{equation}
and by evaluating the powers of $X'B^{-1}$, equation (\ref{tritis})
gives:
$$ X'B^{-1}[f_{2}^{2}L-f_{2}f_{1}M+(f_{1}^{2}-f_{2}f_{0})B^{3}]=
f_{2}^{2}K-f_{2}f_{0}M+f_{1}f_{0}B^{3}
$$
(to alleviate notation we have dropped the $X$ dependence from $f_{i}(X)$ and denote it simply by  $f_{i}$,
for $i=0,1,2$).The last equation can be written as
\begin{equation} \label{tel}
X'B^{-1}[f_{2}^{2}L-f_{2}f_{1}M+(f_{1}^{2}-f_{2}f_{0})B^{3}]=XB^{-1}[f_{2}^{2}L-f_{2}f_{1}M+(f_{1}^{2}-f_{2}f_{0})B^{3}]+A
\end{equation}
where
$$
A=f_{2}^{2}K-f_{2}f_{0}M+f_{1}f_{0}B^{3}-XB^{-1}[f_{2}^{2}L-f_{2}f_{1}M+(f_{1}^{2}-f_{2}f_{0})B^{3}].
$$
Replacing again $K,L,M$ by $XYZ$, $XYB+XBZ+BYZ$, $XB^{2}+BYB+B^{2}Z$ respectively we can factorize $A$ as follows
$$A=(f_{2}(XB^{-1})^{2}-f_{1}XB^{-1}+f_{0})(f_{1}B^{3}-f_{2}B^{2}Z-f_{2}BYB)$$
and from Cayley--Hamilton theorem we get that $A=0$. So from
(\ref{tel}), for the generic elements $X, \ Y, \ Z \in
GL_{2}(\mathbb{C})$ such that
$\det[f_{2}^{2}L-f_{2}f_{1}M+(f_{1}^{2}-f_{2}f_{0})B^{3}] \neq 0$ we
have that $X'=X$. In a similar way we can prove that $Z'=Z$ and finally $Y'=Y$ follows as well.

\end{document}